\documentclass[reqno,11pt,a4paper]{amsart}

\usepackage{tikz}
\usetikzlibrary{matrix,arrows,calc,snakes,patterns,decorations.markings}

\usepackage[mathscr]{eucal}
\usepackage{graphics,epic}
\usepackage{amsfonts, mathtools}
\usepackage{amscd}
\usepackage{latexsym}
\usepackage{amsmath,amssymb, amsthm, stmaryrd, bm, bbm}
\usepackage[all,2cell]{xy}
\usepackage{mathrsfs}
\usepackage{url, hyperref}

\hypersetup{
    colorlinks,
    linkcolor={red!50!black},
    citecolor={blue!50!black},
    urlcolor={blue!80!black}
}

\newtheorem{thm}{Theorem}[section]
\newtheorem*{thm*}{Theorem}

\newtheorem{lem}[thm]{Lemma}
\newtheorem{prop}[thm]{Proposition}
\newtheorem{cor}[thm]{Corollary}

\newtheorem*{conj*}{Conjecture}
\newtheorem{question}[thm]{Question}
\newtheorem*{question*}{Question}
\theoremstyle{remark}
\newtheorem{rem}[thm]{Remark}

\theoremstyle{definition}
\newtheorem{defn}[thm]{Definition}

\newcommand{\MCM}{\opname{MCM}\nolimits}
\newcommand{\krdim}{\opname{kr.dim}\nolimits}

\newcommand{\ul}[1]{\underline{#1}}

\newcommand{\opname}[1]{\operatorname{\mathsf{#1}}}

\renewcommand{\mod}{\opname{mod}\nolimits}

\newcommand{\Z}{\mathbb{Z}}

\newcommand{\C}{\mathbb{C}}

\newcommand{\id}{\mathrm{id}}

\newcommand{\Hom}{\opname{Hom}}

\newcommand{\Ext}{\opname{Ext}}
\newcommand{\ExtCheck}{\opname{\check{E}xt}}
\newcommand{\ExtTV}{\opname{\widehat{Ext}}}

\newcommand{\GL}{\opname{GL}}

\newcommand{\Perf}{\opname{Perf}}

\newcommand{\gldim}{\opname{gldim}\nolimits}

\newcommand{\cf}{{\mathcal F}}

\newcommand{\del}{\partial}

\newcommand{\mm}[1]{#1}

\numberwithin{equation}{section}

\begin{document}

\title{Classifying dg-categories of matrix factorizations}

\author{Martin Kalck}
\email{martin.maths@posteo.de}

\begin{abstract}
We give a complete classification of differential $\Z$-graded homotopy categories of matrix factorizations of isolated singularities up to quasi-equivalence. This answers a question of Bernhard Keller and Evgeny Shinder.

More generally, we show that a quasi-equivalence between the dg singularity category of a Gorenstein isolated singularity $R$ and the dg singularity category of a complete local Noetherian $\C$-algebra $S$ of \emph{different} Krull dimension can always be realized by Kn{\"o}rrer's periodicity -- in particular, the existence of such an equivalence implies that $R$ and $S$ are hypersurface singularities. This uses and is complemented by a recent categorical version of the Mather--Yau theorem for hypersurfaces of the same Krull dimension due to Hua \& Keller, which completes the classification mentioned above.
\end{abstract}

\maketitle

\thispagestyle{empty}

\section{Introduction}
\noindent
Matrix factorizations can be traced back to Dirac's seminal description of the electron taking both quantum theory and relativity into account, cf.~\cite{Dirac} and also \cite{Murfet}. More recently, they have been studied in relation with, for example, Landau-Ginzburg models in homological mirror symmetry \cite{Orlov2}, the operation of flops in the minimal model program in birational geometry \cite{CurtoMorrison, Wemyss18}, curve-counting invariants \cite{BrownWemyss} used in enumerative geometry and    
the construction of new knot invariants \cite{KR}, cf. also \cite{Murfet}.

Matrix factorizations of a power series $f \in P_n=\C\llbracket z_0, \ldots, z_n\rrbracket$ can be considered as objects of the \emph{homotopy category of matrix factorizations} $[\mathsf{MF}(P_n, f)]$, which has a triangulated structure \cite{Buchweitz, Eis80}. 
A fundamental result in the theory of matrix factorizations relates these categories for the power series $f \in P_n$ and $f + z_{n+1}^2 + z_{n+2}^2 \in P_{n+2}$, \cite{KnorrerCMmodules}.
\begin{thm}[Kn{\"o}rrer 1987]\label{T:Knoerrer}
For $n \in \Z_{\geq 0}$ and $0 \neq f \in P_n$ there is a triangle equivalence
\begin{align}\label{E:Knoerrer}
[\mathsf{MF}(P_n, f) ] \xrightarrow{\sim} [\mathsf{MF}(P_{n+2}, f + z_{n+1}^2 + z_{n+2}^2) ]. 
\end{align}
\end{thm}

\noindent
In particular, this gives a bijection between matrix factorizations of $f$ and $f + z_{n+1}^2 + z_{n+2}^2$ up to sums of trivial matrix factorizations. 

Homotopy categories of matrix factorizations fit into the more general framework of \emph{triangulated singularity categories} $D^{sg}(R) = D^b(R)/\mathsf{Perf}(R)$, by Buchweitz \cite{Buchweitz}, Eisenbud~\cite{Eis80}.

\begin{thm}
\label{T:BuchweitzEisenbud}
For $n \in \Z_{\geq 0}$ and $0 \neq f \in P_n$ there is a triangle equivalence
\begin{align}\label{E:BE}
[\mathsf{MF}(P_n, f) ] \xrightarrow{\sim} D^{sg}(P_n/(f)).
\end{align}
\end{thm}

\noindent
The singularity category $D^{sg}(R)$ admits a canonical dg enhancement given by the dg quotient category $D^{sg}_{dg}(R)=D^b_{dg}(R)/\Perf_{dg}(R)$ \cite{Keller, Drinfeld}, where $D^b_{dg}(R)$ is the canonical dg enhancement of $D^b(R)$, which induces a dg enhancement $\Perf_{dg}(R)$ of $\Perf(R)$. There is a quasi-equivalence of dg categories lifting the triangle equivalence \eqref{E:BE}, cf. e.g. \cite[6.6.4]{Booth}
\begin{align}\label{E:BElift}
\mathsf{MF}(P_n, f)  \xrightarrow{\sim} D^{sg}_{dg}(P_n/(f)).
\end{align}
Using this one can show that Kn{\"o}rrer's equivalences \eqref{E:Knoerrer} can be lifted to quasi-equivalences between dg singularity categories. 
These equivalences preserve the parity of the Krull dimension of the singularities, which leads to the following natural question \cite{KellerEmail}.

\begin{question}[Keller \& Shinder]\label{Q:KellerShinder} Assume that $f \in P_n$ defines an isolated singularity and let $g \in P_m$ with $n \not\equiv m (\mod 2)$. Do quasi-equivalences $D^{sg}_{dg}(P_n/(f)) \cong D^{sg}_{dg}(P_{m}/(g))$ exist ?
\end{question}

\noindent
Our main theorem gives a negative answer to this question.

More generally, we give a complete classification of singularities $S$, which have dg singularity categories admitting quasi-equivalences to $D^{sg}_{dg}(P_n/(f))$.

\begin{thm}\label{T:MAIN}
Let $R \cong P_d/(f)$ be an isolated hypersurface singularity and let $S$ be a commutative complete local Noetherian $\C$-algebra of Krull dimension $e$.

Then the following statements are equivalent.
\begin{itemize}
\item[(a)] There is a $\C$-linear quasi-equivalence between dg singularity categories 
\begin{align}\label{E:Quasi}
D_{dg}^{sg}(R) \cong D_{dg}^{sg}(S).
\end{align}
\item[(b)] There is an $n \in \Z_{\geq 0}$ and an algebra isomorphism $S \cong P_e/(g)$, such that
\begin{align}
\lvert d-e \rvert=2n  \qquad \text{and} \qquad   g -f = z_1^2 + \cdots + z_{2n}^2. 
\end{align}
In particular, $S$ and $R$ are stably equivalent singularities in the sense of Arnol'd.
\end{itemize}
\end{thm}

\begin{rem}
Theorem \ref{T:MAIN} shows in particular that there are only countably many isomorphism classes of commutative complete Noetherian $\C$-algebras with dg singularity categories quasi-equivalent to dg singularity categories $D_{dg}^{sg}(P_d/(f))$ for isolated hypersurface singularities $f$. 

This is in stark contrast to the non-commutative case: there is an uncountable family of pairwise non-Morita equivalent complete Noetherian $\C$-algebras with singularity categories triangle equivalent to $D^{sg}(P_{2d}/(f))$ for all ADE-singularities $f \in P_{2d}$ except $E_8$, see \cite{KIWY15}. For $A_1$-singularities $f=x_0^2 + \cdots + x_{2d}^2$, it is known (e.g.~by \cite{KW}) that these triangle equivalences lift to quasi-equivalences between dg singularity categories. 
\end{rem}

For singularities of different Krull dimensions this result can be improved.

\begin{cor}\label{C:MAIN}
Let $R=P_n/I$ be a complete local isolated Gorenstein singularity and let $S$ be a commutative complete local Noetherian $\C$-algebra such that \begin{align} \krdim S \neq \krdim R. \end{align} If there is a $\C$-linear \emph{triangle} equivalence
\begin{align}\label{E:Tria}
D^{sg}(R) \cong D^{sg}(S).
\end{align}
then $R\cong P_d/(f)$ and $S \cong P_e/(g)$ are hypersurfaces. 

In particular, the statements (a) and (b) of Theorem \ref{T:MAIN} are also equivalent for Gorenstein algebras $R=P_n/I$, provided that $\krdim S \neq \krdim R$.
\end{cor}

\begin{rem}
Question \ref{Q:KellerShinder} has a positive answer for certain \emph{group graded} singularity categories of some hypersurface singularities \cite{HIMO} and for certain \emph{non}-Gorenstein cyclic quotient singularities in dimension $2$ and $3$, see \eqref{E:K21}. 
\end{rem}

\noindent
We give a list of all known (cf. \cite{Matsui}) non-trivial triangulated singular equivalences between commutative complete local Noetherian $\C$-algebras. We use the following notation: for a primitive $n$th root of unity $\epsilon_n \in \C$, we define cyclic subgroups of order $n$ in $\GL(\mm{m}, \C)$
\begin{align}
{\frac{1}{n}(a_1, \ldots, a_m)}=\left\langle \mathsf{diag}\left(\epsilon_n^{a_1}, \ldots \epsilon_n^{a_m}\right)
\right\rangle \subset \GL(\mm{m}, \C), \text{ where } a_i \in \Z_{>0}.
\end{align}
The invariant rings under the diagonal action on $\C\llbracket z_1, \ldots, z_m \rrbracket$ are denoted by
\begin{align}
\C\llbracket z_1, \ldots, z_m \rrbracket^{\frac{1}{n}(a_1, \ldots, a_m)}.
\end{align}

\begin{itemize}
\item[(a)] Iterating Kn{\"o}rrer's equivalences \eqref{E:Knoerrer} for $0 \neq f \in P_n$ yields
\begin{align}
D^{sg}(P_{n}/(f)) \cong D^{sg}(P_{n+2m}/(f+z_1^2 + \ldots + z_{2m}^2)).
\end{align}
These are the only known non-trivial equivalences involving Gorenstein singularities -- which in view of Theorem \ref{T:MAIN} and Corollary \ref{C:MAIN} is maybe not so surprising. 
\item[(b)] Relative singularity category techniques \cite{KY16, KY18}, yield singular equivalences \cite{YangPrivate}  \begin{align}\label{E:Yang}
D^{sg}\left(\C\llbracket y_1, y_2\rrbracket^{\frac{1}{n}(1, 1)}\right) \cong D^{sg}\left(\frac{\C[z_1, \ldots, z_{n-1}]}{(z_1, \ldots, z_{n-1})^2}\right).
\end{align}
These equivalences can also be deduced from \cite{Kawamata15}.
In \cite{KK17}, we give another proof of \eqref{E:Yang} \mm{and also construct (noncommutative) finite dimensional algebras $K_{n, a}$} that generalize \mm{\eqref{E:Yang}} to all cyclic quotient surface singularities \mm{$\C\llbracket y_1, y_2\rrbracket^{\frac{1}{n}(1, a)}$}.
\item[(c)] The following are the only known singular equivalences that do not preserve the parity of the Krull dimension, see \cite{K21}.
\begin{align}\label{E:K21}
D^{sg}\left(\C\llbracket x_1, x_2, x_3\rrbracket^{\frac{1}{2}(1, 1, 1)}\right) \cong D^{sg}\left(\C\llbracket y_1, y_2\rrbracket^{\frac{1}{4}(1, 1)}\right) \cong  D^{sg}\left(\frac{\C[z_1, \ldots, z_{3}]}{(z_1, \ldots, z_{3})^2}\right).
\end{align}
\end{itemize}

\begin{rem}
The triangle equivalences listed above can be lifted to quasi-equivalences between dg singularity categories, cf. Theorem \ref{T:MAIN} for (a) and \cite{KW} for (b) and (c). We do not know any examples of singularity categories, which are triangle equivalent and \emph{not} quasi-equivalent as dg categories.
\end{rem}

\noindent
\emph{Acknowledgement.} I am very grateful to Bernhard Keller and Evgeny Shinder for asking me the question that started this work. Moreover, discussions with Evgeny Shinder were of great help for proving the main theorem. I would like to thank the anonymous referee of \cite{K21} for pointing out the graded singular equivalences in \cite{HIMO}, which led me to Proposition \ref{P:Knoerrer}. I am grateful to Nils Carqueville and Zhengfang Wang for feedback on a preliminary version of this article.

\section{Consequences of triangle equivalences between singularity categories}

\noindent
In this section, we collect results about singularity categories and \emph{triangle} equivalences between them.
These statements are either known or follow from combinations of well-known results. Our main references for singularity categories is \cite{Buchweitz}, see also \cite{Yoshino} for Gorenstein singularities.

The triangulated singularity category detects Gorenstein isolated singularities.

\begin{thm}\label{T:Auslander}
Let $(R, \mathfrak{m}, k)$ be a complete local commutative Noetherian $k$-algebra with $k=R/\mathfrak{m}$. 
Then the following statements are equivalent.
\begin{enumerate}
\item[(a)] $R$ is Gorenstein and has an isolated singularity.
\item[(b)] $D^{sg}(R)$ is Hom-finite over $k$, i.e. $\Hom_{D^{sg}(R)}(M, N)$ is a finite dimensional $k$-vector space for all $M, N \in D^{sg}(R)$.
\end{enumerate}
Moreover, if these equivalent conditions are satisfied, then the category $D^{sg}(R)$ has a Serre functor, which is given by the shift functor $[d-1]$.
\end{thm}
\begin{proof}
If $R$ is Gorenstein, Buchweitz shows a triangle equivalence \cite{Buchweitz} 
\begin{align}\label{E:Buchweitz}
D^{sg}(R) \cong \ul{\MCM}(R).
\end{align}
By work of Auslander \cite{Aus84}, the latter category is Hom-finite if and only if $R$ has an isolated singularity. This shows that (a) implies (b) and in view of \eqref{E:Buchweitz} it also shows that for the implication $(b) \Rightarrow (a)$ it is enough to prove that $R$ is Gorenstein. 

Let $\ExtTV^0_R(k, k)$ be the stable cohomology (in the sense of \cite{AV}) of the $R$-module $k=R/\mathfrak{m}$. There is an isomorphism of $k$-vector spaces
\begin{align}
\ExtTV^0_R(k, k) \cong \Hom_{D^{sg}(R)}(k, k),   
\end{align} 
cf.~\cite[1.4.2]{AV}\footnote{More precisely, for finitely generated $R$-modules $M, N$ there are isomorphisms 
\begin{align}
\Hom_{D^{sg}(R)}(M, N) \cong \lim_n \ul{\Hom}_R(\Omega^n(M), \Omega^n(N)) \cong \lim_n \Ext^n_R(M, \Omega^n(N)) \cong \ExtTV^0_R(M, N), 
\end{align}
where the first isomorphism follows from \cite[Corollary 3.9(1)]{Bel00} and the last isomorphism is \cite[Lemma 5.1]{YoshinoHalf}, which uses the notation $\ExtCheck^i_R(M, N)$ for the stable cohomology $\widehat{\Ext}^i_R(M, N)$ and $\Omega_n$ for the $n$th syzygy.
}. Therefore, if $D^{sg}(R)$ is Hom-finite, then $\ExtTV^0_R(k, k)$ is a finite dimensional $k$-vector space, which shows that $R$ is Gorenstein by \cite[6.4]{AV}.

The statement about the Serre functor follows from \eqref{E:Buchweitz} and \cite{Aus78}. 
 \end{proof}
 
 \begin{lem}\label{L:Cohen}
Let $(R, \mathfrak{m})$ be a complete local commutative Noetherian $\C$-algebra. 

Consider the following statements.
\begin{itemize}
\item[(a)] There is a $\C$-linear equivalence $D^{sg}(R) \cong D^{sg}(S)$, where $S \cong P_n/I$ defines an isolated Gorenstein singularity.
\item[(b)] $D^{sg}(R)$ is Hom-finite over $\C$. 
\item[(c)] $R/\mathfrak{m} \cong \C$.
\item[(d)] $R \cong P_m/J$.
\end{itemize}
The following implications hold: $(a) \Rightarrow (b)$ and  $(c)  \Rightarrow (d)$. Moreover,  if $\gldim R = \infty$, then $(b) \Rightarrow (c)$.
\end{lem} 
\begin{proof}
Our assumption $S \cong P_n/I$ implies $S/\mathfrak{m} \cong \C$. In combination with Theorem \ref{T:Auslander} this shows that $(a) \Rightarrow (b)$. We show that $(b)$ \& $\gldim R = \infty$ imply $(c)$. If $\gldim R = \infty$, then $\Hom_{D^{sg}(R)}(R/\mathfrak{m}, R/\mathfrak{m}) \neq 0$. By assumption $(b)$, this is a finite dimensional $\C$-vector space. It is also a $R/\mathfrak{m}$-vector space.
This shows that the field extension $\C \subseteq R/\mathfrak{m}$ is finite, which gives $\C \cong R/\mathfrak{m}$ as $\C$ is algebraically closed.
The implication $(c) \Rightarrow (d)$ is part of the Cohen structure theorem (cf.~e.g.~\cite[\href{https://stacks.math.columbia.edu/tag/032A}{Theorem 032A}]{stacks-project}).
\end{proof}

The singularity category also detects hypersurface singularities.
 
 \begin{thm}\label{T:EisenbudGulliksen}
Let $(R, \mathfrak{m})$ be a complete local Noetherian Gorenstein $\C$-algebra, s.th. $R/\mathfrak{m} \cong \C$.
Then the following statements are equivalent.
\begin{itemize}
\item[(a)] $R \cong P_n/(f)$ is a hypersurface singularity.
\item[(b)] The shift functor of $D^{sg}(R)$ satisfies $[n] \cong \id$ for some $n \in \Z \setminus \{0\}$.
\item[(c)] The shift functor of $D^{sg}(R)$ satisfies $[2] \cong \id$.
\end{itemize}
\end{thm}
\begin{proof}
If $R$ is a hypersurface singularity, then $D^{sg}(R)$ is equivalent to a homotopy category of matrix factorizations by Theorem \ref{T:BuchweitzEisenbud}. By definition, the latter category has a shift functor satisfying $[2] \cong \id$. This shows that (a) implies (c). The implication (c) to (b) is clear.
If the shift functor satisfies $[n] \cong \id$ for some $n \in \Z \setminus \{0\}$, then $R$ is a hypersurface by \cite[Remark 6.5.12]{Booth}, which builds on work of Gulliksen \cite{Gulliksen}. Since $R/\mathfrak{m} \cong \C$, it follows from Lemma \ref{L:Cohen} that $R \cong P_n/(f)$. 
\end{proof}

 \begin{thm}\label{T:Hyper}
Let $(R, \mathfrak{m})$ be a complete local commutative Noetherian Gorenstein $\C$-algebra of Krull dimension $d$ with an isolated singularity.
Let $\cf \neq 0$ be a Hom-finite $\C$-linear triangulated category with Serre functor $\mathbb{S}_\cf$ satisfying 
\begin{align}\label{E:fractCY}
\mathbb{S}_\cf^n \cong [m] \quad \text{ for } \ (n, m) \in \Z_{>0} \times \Z \setminus \{(n , n(d-1))\}. 
\end{align}
If there is a $\C$-linear equivalence of triangulated categories
\begin{align}\label{E:fractional}
D^{sg}(R) \cong \cf,
\end{align}
then $R \cong P_d/(f)$ is a hypersurface. 

\end{thm}
\begin{proof}
By \cite{BondalKapranov}, Serre functors are unique up to isomorphism. Combining the equivalence \eqref{E:fractional} with the fact that $D^{sg}(R)$ has Serre functor $[d-1]$ (Theorem \ref{T:Auslander}) shows that $\mathbb{S}_\cf \cong [d-1]$. Now \eqref{E:fractCY} yields natural isomorphisms in $D^{sg}(R)$
\begin{align}
[m] \cong  \mathbb{S}_\cf^n \cong [n(d-1)] \quad \Rightarrow \quad [m - n(d-1)] \cong \id,
\end{align}
where $m - n(d-1) \neq 0$ by assumption. Therefore, the statement follows from Theorem \ref{T:EisenbudGulliksen} -- note that  $R/\mathfrak{m} \cong \C$ by the implication (b) $\Rightarrow$ (c) in
Lemma \ref{L:Cohen}, where we use that $\cf \neq 0$ implies $\gldim R = \infty$.
\end{proof}

Using the results above we can now collect known implications of \emph{triangle} equivalences between singularity categories of commutative $\C$-algebras.  
\begin{cor}\label{C:Hyper}
Let $R=P_n/I$ and let $S$ be a commutative complete local Noetherian $\C$-algebra. If there is a $\C$-linear triangle equivalence 
\begin{align}
D^{sg}(R) \cong D^{sg}(S),
\end{align}
then the following statements hold
\begin{itemize}
\item[(a)] $R$ is Gorenstein and has an isolated singularity if and only if $S$ has these properties.
In particular, if these equivalent conditions are satisfied, then $S \cong P_m/J$.
\item[(b)] If $R$ is a Gorenstein isolated singularity and $\krdim R \neq \krdim S$, then $R$ and $S$ are hypersurfaces.
\item[(c)] $R$ is a isolated hypersurface singularity if and only if $S$ is a isolated hypersurface singularity.
If these equivalent conditions are satisfied, then $S \cong P_m/(g)$.
\end{itemize}
\end{cor}
\begin{proof}
Part (a) follows from the categorical characterization of Gorenstein isolated singularities in Theorem \ref{T:Auslander} together with the implications (a) $\Rightarrow$ (c) \& (d) in Lemma \ref{L:Cohen}.
Part (b) is a consequence of part (a) in combination with Theorem \ref{T:Hyper} applied to $\cf=D^{sg}(S)$ with $(n, m)=(1, e-1)$. Since hypersurface singularities are Gorenstein, part (c) follows from part (a) together with the categorical characterization of hypersurface singularities in Theorem \ref{T:EisenbudGulliksen}. 
\end{proof}

\begin{rem}
More generally, if a complete local Gorenstein algebra $S$ is (triangle) singular equivalent to a complete intersection $R=P_n/(f_1, \ldots, f_c)$ of codimension $c$, then $S$ is isomorphic to a complete intersection $P_n/(g_1, \ldots, g_c)$ of codimension $c$, cf. \cite{Puthenpurakal}.
\end{rem}

\noindent
Theorem \ref{T:Hyper} and Corollary \ref{C:Hyper} also have consequences for the existence of singular equivalences involving finite dimensional associative $\C$-algebras.

\begin{cor}\label{C:NC}
Let $R$ be a commutative complete local Noetherian Gorenstein $\C$-algebra of Krull dimension $d > 0$ and let $A$ be a finite dimensional connected associative $\C$-algebra.
If there is a $\C$-linear triangle equivalence 
\begin{align}
D^{sg}(R) \cong D^{sg}(A),
\end{align}
then the following statements hold
\begin{itemize}
\item[(a)] If $A$ is commutative, then $R$ is a hypersurface and $A \cong \C[x]/(x^n)$.
\item[(b)] If $A$ is symmetric, i.e. $A \cong \Hom_\C(A, \C)$ as $A$-$A$-bimodules, then $R$ is a hypersurface. 
\end{itemize}
\end{cor}
\begin{proof}
Part (a) is a special case of Corollary \ref{C:Hyper} (b). Part (b) follows from Theorem \ref{T:Hyper} (b) applied to $\cf=D^{sg}(A)$, which has Serre functor $[-1]$, see e.g. \cite{KrauseIyengar20}. 
\end{proof}

\noindent
The following result follows from work of Kn{\"o}rrer \cite{KnorrerCMmodules}, cf. \cite[Appendix]{K21} for (a)  $\Rightarrow$ (b).
\begin{prop}\label{P:Knoerrer} 
Let $n \in \Z_{\geq 0}$ and $m \in \Z_{> 0}$. Assume that $0 \neq f \in P_n$ has an isolated singularity. 
Then the following statements are equivalent.
\begin{itemize}
\item[(a)] There is a $\C$-linear triangle equivalence
\begin{align}
D^{sg}(P_{n}/(f)) \cong D^{sg}(P_{n+m}/(f+z_{n+1}^2 + \ldots + z_{n+m}^2)).
\end{align}
\item[(b)] $m$ is even.
\end{itemize} 
\end{prop}
\begin{proof}
The implication (b) $\Rightarrow$ (a) follows by iterating Theorem \ref{T:Knoerrer}. 

To see the other direction, we assume that $m$ is odd. We first note that by Theorem \ref{T:Knoerrer} it is enough to consider the case $m=1$. This case is shown in \cite[Appendix]{K21}. For the convenience of the reader we repeat the argument. The Serre functors $[n-1]$ and $[n]$ of the equivalent categories $D^{sg}(P_{n}/(f))$ and $D^{sg}(P_{n+1}/(f+z_{n+1}^2 ))$ are isomorphic \cite{BondalKapranov}.  This yields the following natural isomorphism in both categories
\begin{align}
[1] \cong \id.
\end{align} 
In particular,
$X[1]\cong X$ for every indecomposable object $X$ in $D^{sg}(P_{n}/(f))$. It follows from
\cite[Prop. 2.7. i)]{KnorrerCMmodules} that there is
an indecomposable matrix factorization $Y$ of $f+z_{n+1}^2$ such that
$Y[1]\ncong Y$. Since $X$ corresponds to a non-trivial matrix factorization, $Y$ is non-trivial by \cite[Lemma 2.5. ii)]{KnorrerCMmodules} and therefore $Y \not\cong 0$ in $D^{sg}(P_{n+1}/(f+z_{n+1}^2 ))$. This contradicts $[1] \cong \id$ in
$D^{sg}(P_{n+1}/(f+z_{n+1}^2 ))$ and shows that (a) is impossible.
\end{proof}

\section{Proof of Theorem \ref{T:MAIN} and Corollary \ref{C:MAIN}}

\begin{defn}
Let $R=P_n/(f)$ be a hypersurface singularity. The algebra 
\begin{align}
T_R = P_n/(f, \del_0 f, \ldots, \del_n f)
\end{align}  
is called the \emph{Tyurina algebra} of $R$. 
\end{defn}

By definition, the Tyurina algebra is invariant under changing a hypersurface by adding squares in additional variables.
\begin{lem}\label{L:Tyurina}
Let $R=P_n/(f)$ and $S=P_m/(f+z_{n+1}^2 + \ldots + z_m^2)$ be hypersurface singularities. Then there is an isomorphism of algebras
\begin{align}
T_R \cong T_S.
\end{align}
\end{lem}

We are ready to prove Theorem \ref{T:MAIN}.
\begin{proof}
We start with showing that (a) implies (b). The quasi-equivalence between dg singularity categories \eqref{E:Quasi} yields a triangle equivalence between singularity categories 
\begin{align}\label{E:SingTriang}
D^{sg}(R) \cong D^{sg}(S).
\end{align}
Since $R$ is an isolated hypersurface singularity the same is true for $S$ by Corollary \ref{C:Hyper} (c).
Now, the quasi-equivalence \eqref{E:Quasi} between dg singularity categories of hypersurfaces $R$ and $S \cong P_e/(g')$, yields an isomorphism between their Tyurina algebras
\begin{align}\label{E:Tyurina}
T_R \cong T_{S},
\end{align}
see \cite{HuaKeller} cf. also \cite[Theorem 6.6.11]{Booth}.

Without loss of generality, assume that $e\geq d$ and define $R'=P_e/(f + z_{d+1}^2 + \cdots z_{e}^2)$. By Lemma \ref{L:Tyurina}, we have \begin{align}\label{E:Tyurina2}
T_R \cong T_{R'}.
\end{align}
Now $R'$ and $S$ are complete local $\C$-algebras of the same Krull dimension $e$, with isomorphic Tyurina algebras (by \eqref{E:Tyurina} and \eqref{E:Tyurina2}). Therefore, the formal version \cite[Prop. 2.1.]{GreuelPham} of the Mather--Yau Theorem \cite{MY}, yields an isomorphism 
\begin{align}
P_e/(f + z_{d+1}^2 + \cdots z_{e}^2) = R' \cong S
\end{align} 
So we can set $g=f + z_{d+1}^2 + \cdots z_{e}^2$.

It remains to show that $e-d$ is even. Assume that $e-d>0$ is odd. Using that $R' \cong S$, \eqref{E:SingTriang}
and Kn{\"o}rrer periodicity \eqref{E:Knoerrer}, we know that 
\begin{align*}
D^{sg}\left(\frac{P_e}{(f + z_{d+1}^2 + \cdots z_{e}^2)}\right) \cong D^{sg}(R') \cong D^{sg}(S) \cong D^{sg}(R) \cong D^{sg}\left(\frac{P_{e-1}}{(f + z_{d+1}^2 + \cdots z_{e-1}^2)}\right)
\end{align*}
But Proposition \ref{P:Knoerrer} shows that a triangle equivalence between singularity categories of hypersurface singularities $P_{n}/(h +z_n^2)$ and $P_{n-1}/(h)$ cannot exist. This contradiction shows that $e-d$ is even and completes the proof of the implication (a) $\Rightarrow$ (b).

The other implication (b) $\Rightarrow$ (a) follows from Kn{\"o}rrer's periodicity result \cite{KnorrerCMmodules}, which is induced by a quasi-equivalence between the corresponding dg singularity categories by work of Orlov \cite{Orlov2, OrlovIdempotent}, cf.~e.g.~\cite[Theorems 1.3 and 1.5]{PavicShinder}.
\end{proof}

Corollary \ref{C:MAIN} follows from Corollary \ref{C:Hyper} (b) in combination with Theorem \ref{T:MAIN}.


\begin{thebibliography}{40}
\mm{\bibitem{Aus78} M.~Auslander, \emph{Functors and morphisms determined by objects}, Representation
  theory of algebras (Proc. Conf., Temple Univ., Philadelphia, Pa., 1976),
  Dekker, New York, 1978, pp.~1--244. Lecture Notes in Pure Appl. Math., Vol.
  37.}
   \bibitem{Aus84}
\mm{\bysame, \emph{{Isolated singularities and existence of almost split
  sequences. Notes by Louise Unger.}}, {Representation theory II, Groups and
  orders, Proc. 4th Int. Conf., Ottawa/Can. 1984, Lect. Notes Math. 1178,
  194-242 (1986)}.}
 \bibitem{AV} L. Avramov, O.~Veliche, \emph{Stable cohomology over local
  rings}, Adv. Math. \textbf{213} (2007), no.~1, 93--139. 
\bibitem{Bel00} A.~Beligiannis, \emph{The homological theory of contravariantly finite subcategories: Auslander-Buchweitz contexts, Gorenstein categories and (co-)stabilization}, 
Comm. in Algebra \textbf{28} (10), 4547--4596, 2000.
\mm{\bibitem{BondalKapranov} A.~Bondal, M.~Kapranov, \emph{Representable functors, Serre functors, and reconstructions}, Izv. Akad. NaukSSSR Ser. Mat. 53 (1989), no. 6, 1183--1205.}
\bibitem{Booth} M.~Booth, \emph{The derived contraction algebra}, PhD thesis, Edinburgh, 2019, arXiv:1911.09626.
\bibitem{BrownWemyss} G. Brown, M. Wemyss, \emph{Gopakumar--Vafa invariants do not determine flops}, Comm. Math. Phys. 361 (2018), no. 1, 143--154.
\bibitem{Buchweitz} R.-O.~Buchweitz, \emph{Maximal Cohen-Macaulay modules and Tate-Cohomology over
Gorenstein rings}, Preprint 1986, available at \url{http://hdl.handle.net/1807/16682}.
\bibitem{CurtoMorrison} C. Curto, D. R. Morrison, \emph{Threefold flops via matrix factorization}, J. Algebraic Geom.,
22(4):599--627, 2013.

\bibitem{Dirac}
P.~A.~M. Dirac,
\newblock \emph{The quantum theory of the electron},
\newblock { Proceedings of the Royal Society of London A: Mathematical,
  Physical and Engineering Sciences}, 117(778):610--624, 1928.
\bibitem{Drinfeld} V.~Drinfeld, \emph{DG quotients of DG categories}, J. Algebra 272 (2004), no. 2, 643--691.  
\mm{\bibitem{Eis80}
D.~Eisenbud, \emph{Homological algebra on a complete intersection, with an application to group representations}, Trans.\ Amer.\ Math.\ Soc.\ {\bf 260} (1980) 35--64.}
\bibitem{GreuelPham}  G.-M.~Greuel, T.~H.~Pham, \emph{Mather-Yau theorem in positive characteristic}, J. Algebraic Geom. \textbf{26} (2017), no. 2, 347--355.
\bibitem{Gulliksen} T.~Gulliksen, \emph{A proof of the existence of minimal R-algebra resolutions}, Acta Math.120 (1968), 53--58.
\bibitem{HIMO} M. Herschend, O. Iyama, H. Minamoto, S. Oppermann, \emph{Representation theory of Geigle--Lenzing complete intersections}, Mem. Amer. Math. Soc. (to appear), arXiv:1409.0668.
\bibitem{HuaKeller} Z.~Hua, B.~Keller, \emph{Cluster categories and rational curves}, arXiv:1810.00749.
\bibitem{KrauseIyengar20} S.~Iyengar, H.~Krause, \emph{The Nakayama functor and its completion for Gorenstein algebras}, 
arXiv:2010.05676.
  \bibitem{K21} M.~Kalck, \emph{A new equivalence between singularity categories of commutative algebras}, \\
  Advances in Mathematics,
vol.~\textbf{390}, 107913 (October 2021), see also arxiv:2103.06584
   \href{https://arxiv.org/pdf/2103.06584.pdf}{(pdf)}. 
  \bibitem{KIWY15} M.~Kalck, O.~Iyama, M.~Wemyss, D.~Yang, \emph{Frobenius categories, Gorenstein algebras and rational surface singularities},   
 Compositio Mathematica, vol. \textbf{151}, issue 03, 502--534 (2015).
\bibitem{KK17} M.~Kalck, J.~Karmazyn, \emph{Noncommutative Kn{\"o}rrer type equivalences via noncommutative resolutions of singularities},
arXiv:1707.02836, \href{https://arxiv.org/pdf/1707.02836.pdf}{(pdf)}.
\bibitem{KW} M.~Kalck, Z.~Wang, unpublished notes.
\bibitem{KY16} M.~Kalck, D.~Yang, \emph{Relative singularity categories I: Auslander resolutions}, \\
  Advances in Mathematics, vol. \textbf{301}, 973--1021 (2016), \href{https://arxiv.org/pdf/1205.1008.pdf}{(pdf)}.  
\bibitem{KY18} M.~Kalck, D.~Yang, \emph{Relative singularity categories II: DG models},   
  arXiv1803.08192, \href{https://arxiv.org/pdf/1803.08192.pdf}{(pdf)}. 
  \bibitem{Kawamata15}
Y.~Kawamata,
\newblock \emph{On multi-pointed noncommutative deformations and Calabi-Yau threefolds}, Compos. Math. 154 (2018), no. 9, 1815--1842.
 \bibitem{Keller} B.~Keller, \emph{On the cyclic homology of exact categories}, J. Pure Appl. Algebra 136 (1999), no. 1, 1--56.
\bibitem{KellerEmail} B.~Keller, Email to the author on April 27, 2020. 
\bibitem{KR}
M.~Khovanov, L.~Rozansky,
\newblock \emph{Matrix factorizations and link homology},
\newblock {Fund. Math.}, 199(1):1--91, 2008.
\bibitem{KnorrerCMmodules}
H.~Kn{\"o}rrer, \emph{Cohen-{M}acaulay modules on hypersurface singularities {I}}, Invent. Math. 88(1):153--164, 1987.
\bibitem{MY} J.~N.~Mather, S.~S.~T.~Yau, \emph{Classification of isolated hypersurface singularities by their
moduli algebras}, Invent. Math. \textbf{69} (1982), no. 2, 243--251.
\bibitem{Matsui} H.~Matsui, \emph{Singular equivalences of commutative noetherian rings and reconstruction of singular loci}, J. Algebra 522 (2019): 170--194.
\bibitem{Murfet}
D.~Murfet,
\newblock \emph{Matrix factorizations},
\newblock MSRI: Emissary Newsletter, Spring 2013,
\newblock Available at
  \url{https://www.msri.org/attachments/media/news/emissary/EmissarySpring2013.pdf}.
\bibitem{Orlov2} D.~Orlov, 
\emph{Triangulated categories of singularities and {D}-branes in {L}andau-{G}inzburg models},
Tr. Mat. Inst. Steklova \textbf{246} (2004), Algebr. Geom. Metody, Svyazi i Prilozh., 240--262.
\bibitem{OrlovIdempotent}
\bysame,
\newblock \emph{Formal completions and idempotent completions of triangulated
  categories of singularities},
\newblock {Adv. Math.}, \textbf{226}(1):206--217, 2011.
\bibitem{PavicShinder} N.~Pavic, E.~Shinder, \emph{K-theory and the singularity category of quotient singularities}, arXiv:1809.10919. 
\bibitem{Puthenpurakal} T.~Puthenpurakal, \emph{On the Stable category of maximal Cohen--Macaulay modules over Gorenstein rings-I}, arXiv:2107.05237.
\bibitem{stacks-project} {The {Stacks project authors}}, \emph{The Stacks project}, {\url{https://stacks.math.columbia.edu}},
(2021).
\bibitem{Wemyss18}
M.~Wemyss, \emph{Flops and clusters in the homological minimal program}, Invent. Math. \textbf{211} (2018), 435--521.
\bibitem{YangPrivate} D.~Yang, Private Communication, Jan 2015.
\bibitem{Yoshino} Y. Yoshino, \emph{Cohen--Macaulay modules over Cohen--Macaulay rings}, London Mathematical Society Lecture Note Series, 146, Cambridge University Press, Cambridge, 1990.
\bibitem{YoshinoHalf} \bysame, \emph{Tate--Vogel completions of half-exact functors}, Algebras and representation theory, \textbf{4} (2), 171--200, 2001.

\end{thebibliography}
\end{document}